\documentclass[a4paper,12pt]{article}

  \usepackage[top=1in,bottom=1in,left=1in,right=1in]{geometry}          
  \usepackage{amsmath}                                  
  \usepackage{amsthm}                                   

  \usepackage{amssymb}                              
  \usepackage{mathrsfs}                                 
  \usepackage{makeidx}                                  
  \usepackage{color}                                
  \usepackage[pdfborder={0 0 0},colorlinks=true,linkcolor=blue]{hyperref}

  \usepackage{xr}                                                       

  \usepackage[all]{xy}                                

  \usepackage{enumerate}                            

  \numberwithin{equation}{section}
  \theoremstyle{definition}     \newtheorem{Defi}{Definition}[section]
  \theoremstyle{plain}          \newtheorem{theorem}[Defi]{Theorem}
  \theoremstyle{plain}          
  \theoremstyle{plain}          
  \theoremstyle{plain}          
  \theoremstyle{remark}         \newtheorem{example}[Defi]{Example}

  \newcommand{\lgg}{\left| \!\left|}            
  \newcommand{\rgg}{\right| \!\right|}          
  \newcommand{\ldg}{\left|}                     
  \newcommand{\rdg}{\right|}                    
  \newcommand{\wtilde}{\widetilde}              

  \newcommand{\si}{\sigma}

  \newcommand{\e}{\varepsilon}
  \newcommand{\de}{\delta}
  
  \newcommand{\lam}{\lambda}

  \newcommand{\ds}[1]{\mathbb{#1}}              

  \newcommand{\dtp}{d^{\frac1{2p}}}                        

  \author{Qin Zhang \footnote{\small\textit{{Address: \em 3-8-1} Komaba, Meguro, Tokyo, {\em 153-8914}, Japan.}
  \small\textit{The author is supported by the Japanese Government Scholarship (No.052111). Email Address: zqhustwhu@hotmail.com}}}
  \title{Noncommutative maximal ergodic inequality for non-tracial $L^{1}$-spaces}
  \date{\small\textit{Graduate School of Mathematical Sciences, The University of Tokyo}  }

\begin{document}
  \setlength{\baselineskip}{3.5ex plus 0.3ex minus 0.1ex}     

  \maketitle

  \begin{abstract}
    We extend the noncommutative $L^1$-maximal ergodic inequality for semifinite von Neumann algebras established by Yeadon in 1977 to the framework of noncommutative $L^1$-spaces associated with $\sigma$-finite von Neumann algebras. Since the semifinite case of this result is one of the two essential parts in the proof of noncommutative maximal ergodic inequality for tracial $L^p$-spaces $(1<p<\infty)$ by Junge-Xu in 2007, we hope our result will be helpful to establish a complete noncommutative maximal ergodic inequality for non-tracial $L^p$-spaces in the future.
  \end{abstract}

  \section{Introduction}
  \indent\quad Theory of von Neumann algebras is regarded as the noncommutative measure and integration theory (Chapter IX, \cite{30}), so it is natural to consider extensions of classical ergodic theorems for spaces of measurable functions to the framework of noncommutative spaces associated with von Neumann algebras. Such an extension topic appeals to many mathematicians and they had interesting results even from 1970's (for example, the pioneering works of \cite{19} and \cite{32}). Since then many classical mean ergodic theorems and ergodic theorems of other types were successfully transformed to the noncommutative context which is a semifinite von Neumann algebra or a non-commutative $L^{p}$-space associated with a semifinite von Neumann algebra, and some authors even considered a general von Neumann algebra or a non-tracial $L^{p}$-space associated with it. Among these noncommutative ergodic theorems, however, the problem of finding a noncommutative analogue of the famous Dunford-Schwartz maximal ergodic inequality (\cite{8}) was left open until the appearance of Junge-Xu's prominent paper in 2007. The main obstacle in this problem is that it is difficult to define the supremum of a sequence of operators even in the finite-dimensional Hilbert space cases, although it is straightforward to take the supremum of a sequence of measurable functions. For this reason, many powerful techniques in classical ergodic theory involving maximal functions seem no longer available in the research of noncommutative ergodic results. This difficulty was overcome in Junge-Xu's work (\cite{17}) by using the noncommutative vector valued $L^{p}$-space theory developed by Pisier and Junge (see \cite{23} for the case of hyperfinite von Neumann algebras and \cite{16} for that of general ones). Junge-Xu established noncommutative version of the Dunford-Schwartz maximal ergodic inequality first for non-commutaive $L^{p}$-spaces associated with a semifinite von Neumann algebra, or in other words, tracial $L^{p}$-spaces (Theorems 4.1, \cite{17}). In order to state their result we need some notations as follows.

  Let $M$ be a semifinite von Neumann algebra equipped with a normal semifinite faithful trace $\tau$, and let $L^{p}(M)$ be the associated noncommutative $L^{p}$-space (see, for example, \cite{21} for a detailed definition). The symbol $T$: $M \rightarrow M$ denotes a linear map which satisfies the following conditions.

 $(J1)$ $T$ is a contraction on $M$: $\lgg Tx\rgg  \leq \lgg x\rgg $ for all $x \in M$, where $\lgg \cdot\rgg $ means the $\infty$-norm, or in other words, the usual operator norm.

 $(J2)$ $T$ is positive: $Tx\geq 0$ if $x\geq 0$.

 $(J3)$ $\tau\circ T\leq \tau$: $\tau(T(x))\leq \tau(x)$ for all $x \in L^{1}(M) \cap M_{+}$.

  \begin{theorem}{ \em(Theorem 4.1 in \cite{17})}.\label{sec:introduction-thm1.1}
    Let $T$ be a linear map on a semifinite von Neumann algebra $M$ with (J1)--(J3), then $T$ extends naturally to a contraction on $L^{p}(M)$ for $1\leq p< \infty$. Put $$S_{n}\equiv S_{n}(T)=\frac{1}{n+1}\sum^{n}_{k=0}T^{k},$$
    then for every $p$, $1 < p < \infty$, we have
    \begin{equation*}
      \lgg \sup_{n}{}^{+} S_{n}(x) \rgg _{p} \leq c_{p} \lgg x\rgg _{p}, \text{ for all } x\in L^{p}(M),
    \end{equation*}
    for some positive constant $c_{p}$ depending only on $p$.

  \end{theorem}

  Note that $\lgg \sup_{n}^{+} S_{n}(x) \rgg _{p}$ is the notation for $\lgg \{S_{n}(x)\}_{n \in \mathbb{N}}\rgg _{L^{p}(M;\;l^{\infty})}$ used in \cite{17}, where $L^{p}(M;\;l^{\infty})$ is the Pisier-Junge's noncommutative vector valued $L^{p}$-space defined by the space of all sequences $x=\{x_{n}\}_{n \in \mathbb{N}}$ in $L^{p}(M)$ $(1\leq p\leq \infty)$ such that each sequence admits a factorization of the following form: there are $a, b \in L^{2p}(M)$ and $y=\{y_{n}\}_{n \in   \mathbb{N}}\subset L^{\infty}(M)=M$ such that $x_{n}=ay_{n}b$ for all $n \in \mathbb{N}$ and the norm is defined by $$\lgg x\rgg _{L^{p}(M;\;l^{\infty})}\equiv \inf\{\lgg a\rgg _{2p}\sup_{n \in \mathbb{N}} \lgg y_{n}\rgg \lgg b\rgg _{2p}\},$$ where the infimum runs over all such factorizations as above.

  For details of such spaces, we refer the reader to \cite{23} and \cite{16}.

  But considering only semifinite von Neumann algebras is not enough sometimes. It was stated in \cite{30} that most of factors arising from physics are of type III, which are of course not semifinite. Another fact is that it was shown by Pisier (\cite{24}) that $OH$ cannot completely embed in a semifinite $L^{1}(M)$. More and more recent works concern the type III case or need results on noncommutative $L^{p}$-spaces associated with a not necessarily semifinite von Neumann algebra.

  For a general von Neumann algebra, there are several equivalent constructions of non-commutative $L^{p}$-spaces associated with it, and the important ones include \cite{2,6,10,13,15,18,31} and they are all based on the Tomita-Takesaki theory. Since a general von Neumann algebra does not necessarily admit a normal semifinite faithful trace, any of these (equivalent) constructions are called non-tracial $L^{p}$-spaces.

  The first non-tracial $L^{p}$-spaces are Haagerup's ones, and just for Haagerup's $L^{p}$-spaces, Junge-Xu established the non-commutative Dunford-Schwartz maximal ergodic inequality as a non-tracial extension of Theorem \ref{sec:introduction-thm1.1} above (see Theorem 7.4, \cite{17}). Their method is to use an early result of Haagerup named reduction method to approximate Haagerup's $L^{p}$-spaces by simifinite ones. We state this theorem here.

  \begin{theorem}{ \em(Theorem 7.4 in \cite{17})}.\label{sec:introduction-thm1.2}
    Suppose $M$ is a ($\sigma$-finite) von Neumann algebra with a normal faithful state $\varphi$, let $T$ be a linear map on $M$ satisfying the following properties (H1)--(H4).
    \begin{enumerate}[{(H}1{)}]
    \item $T$ is a contraction on $M$: $\lgg Tx\rgg  \leq \lgg x\rgg $ for all $x \in M$.
    \item $T$ is completely positive.
    \item $\varphi(T(x))\leq \varphi(x)$ for all $x \in M_{+}$.
    \item $T\circ \sigma_{t}^{\varphi}=\sigma_{t}^{\varphi} \circ T $ for all $t\in \mathbb{R}$.
    \end{enumerate}

    Then $T$ extends naturally to a contraction on Haagerup's $L^{p}(M)$ for $1\leq p <\infty$. Put $$S_{n}\equiv S_{n}(T)=\frac{1}{n+1}\sum^{n}_{k=0}T^{k}.$$
Then for every $p$, $1 < p < \infty$, we have
    \begin{equation*}
      \lgg \sup_{n} {}^{+} S_{n}(x) \rgg _{p} \leq c_{p} \lgg x\rgg _{p}, \text{ for any } x\in L^{p}(M),
    \end{equation*}
    for some positive constant $c_{p}$ depending only on $p$.

   \end{theorem}
  This theorem was also established in \cite{11} (Theorem 7.9).

  We may compare the statement of this theorem with that of Theorem \ref{sec:introduction-thm1.1}. It is clear that the assumption $(J2)$ is reinforced to the assumption $(H2)$, and $(H4)$ is new, the others are natural extensions. Since the modular automorphism group $\sigma_{t}^{\tau}$ induced by a trace $\tau$ is trivial, we do not think the condition $(H4)$ is unnatural any more when we consider the non-tracial cases. However, we feel that the ``complete positivity'' assumption of $(H2)$ seems a little bit stronger. Although it is remarked in \cite{17}, page 425, that the complete positivity is unlikely to be really necessary, they did not give a proof to release this assumption, and they still used complete positivity assumption in \cite{11}.

  We would like to remark that Haagerup's reduction method used in \cite{11} is really powerful to extend the noncommutative martingale inequalities from the tracial cases to the non-tracial ones without any change in the statement form for almost all those results. This is because a conditional expectation in the framework of von Neumann algebras is certainly completely positive (\cite{1} and \cite{T}). In noncommutative ergodic theorems, ``complete positivity'' seems a little restrictive, moreover, many important positive but not completely positive state-preserving transforms describing open quantum evolution are now considered by mathematical physicists (see, for example, \cite{26,27,14,5,29}).

  So we do not think Theorem \ref{sec:introduction-thm1.2} is a complete answer for the problem of non-tracial extension of Theorem \ref{sec:introduction-thm1.1}. Recalling the proof of Theorem \ref{sec:introduction-thm1.1} in \cite{17}, we know that this proof consists of two essential parts, one is the noncommutative maximal ergodic inequality for semifinite $L^1(M)$ established by Yeadon, and the other one is the noncommutative semifinite Marcinkiewicz interpolation theorem in \cite{17}. A natural and possible way of thought for the problem of a complete non-tracial extension of Theorem \ref{sec:introduction-thm1.1} is to get non-tracial extensions of these two parts. The main result of our paper is just to provide the non-tracial extension of maximal ergodic inequality for the non-tracial $L^1(M)$. We have been unable to get a noncommutative Marcinkiewicz interpolation theorem for the non-tracial $L^p$-spaces so far and we will explain the reason in the next section. We hope the result in this paper will be helpful to establish a complete noncommutative maximal ergodic inequality for the non-tracial $L^p$-spaces in the future.

  It is well-known that all the von Neumann algebras encountered in quantum statistical mechanics and quantum field theory are $\sigma$-finite (\cite{4}, p.84), probably for this reason, only the noncommutative $L^p$-spaces associated with a $\sigma$-finite von Neumann algebra are considered in Theorem \ref{sec:introduction-thm1.2} instead of general non-tracial situation and we think this restriction to the $\sigma$-finite case is of full meaning. Our framework is the Hilsum spatial $L^{p}$-spaces associated with a $\sigma$-finite von Neumann algebra which are certainly isometrically isomorphic to Haagerup's ones associated with a $\sigma$-finite von Neumann algebra. On the detailed construction of the Hilsum spatial $L^{p}$-spaces and the close relation between those spaces and complex interpolation theory, we will give a short description on them in the beginning of the next section.

  Let $M$ be a $\sigma$-finite von Neumann algebra represented on a Hilbert space $H$, hence we may assume $M$ admits a normal faithful state $\varphi$, and $\psi$ is a normal faithful state on the commutant $M^{\prime}$ of $M$, let $L^{p}(M ; \psi)$ denote the Hilsum spatial $L^{p}$-spaces with respect to $\psi$. Our main result is the following theorem.

  \begin{theorem} {\em(Theorem \ref{sec:main-part-11} in the next section)}.
    Assume $M$ is a $\sigma$-finite von Neumann algebra which admits a normal faithful state $\varphi$ and $T$ is a linear map on $M$ satisfying the following conditions.

    $(1)$ $T$ is a contraction of $M$, i.e., $\lgg Tx\rgg\leq \lgg x\rgg$ for all $x\in M$.

    $(2)$ $T$ is positive, i.e., $Ty\geq 0$ if $y\in M_+$.

    $(3)$ $\varphi(T(y))\leq \varphi(y)$, for all $y\in M_+$,

    Then $T$ extends to a positive linear contraction on $L^1(M; \psi)$. For any $a \in L^{1}_{+}(M ; \psi)$ and any $\lambda > 0$, and any $n\in \mathbb N$, there exists a projection $e_n \in M$ such that
    \begin{align*}
    e_nS_r(a)e_n\leq \lam e_nde_n \text{ for all }r\in \{0,1,...,n\},
    \end{align*}
    and
    \begin{align*}
    \varphi(\mathbf1-e_n)\leq \frac2{\lam}\int ad\psi,
    \end{align*}
    where $S_r(a)=\frac1{r+1}\sum_{k=0}^rT^k(a)$, $d=\frac{d\varphi}{d\psi}$ is the spatial derivative and $\mathbf 1$ is the identity of $M$.
    Furthermore, for any $a \in L^{1}_{+}(M ; \psi)$, there exists a projection $e\in M$ such that for any $r\in \mathbb N$,
    \begin{align*}
    \int e S_r(a)e d\psi \leq 4\lambda, \text{ and } \varphi(\mathbf 1-e)\leq \frac2{\lam}\int ad\psi.
    \end{align*}
  \end{theorem}

 We recall that in the semifinite case, the $L^{1}$-norm of the tracial $L^{1}$-space $L^{1}(M)$ is defined by $\lgg \cdot\rgg _{1}=\tau (|\cdot|)$, and in the spatial $L^{1}$-space $L^{1}(M ; \psi)$, the $L^{1}$-norm is defined by $\lgg \cdot\rgg _{1}=(\int |\cdot| d\psi)$, and we know that when $M$ is semifinite, $L^{1}(M ; \psi)$ is equivalent to the tracial $L^{1}(M)$, so this is the reason we use $\int \cdot d\psi$ in the statement of our theorem. In the semifinite case, Yeadon's result (\cite{32}) provided a bound on $eS_r(a)e$, and this is because in semifinite case, the elements in $L^p$-spaces are unbounded (also including bounded) operators affiliated with the von Neumann algebra, hence we may cut such unbounded operators by using some projection in the von Neumann algebra. However, for the non-tracial case, elements in the Hilsum (or other equivalent) $L^p$-spaces are never affiliated with the von Neumann algebra $M$, so $eS_r(a)e$ cannot be majorized by any bounded operator, and the result in our theorem is the best one we can hope in the non-tracial case.

 Such a result also has its own value. In the commutative case we know that the following Chebyshev type inequality is quite important in measure theory and probability theory,
 \begin{align*}
 \mu(\{x:\sup_{n \in \mathbf N}s_n(f)(x)>\lambda\})\leq \frac{c}{\lambda}\lgg f\rgg _1,
 \end{align*}
 and our result is a non-commutative extension of this inequality for $\sigma$-finite von Neumann algebras.

  \section{Main Part}
  \indent\quad Now we introduce the definition of the Hilsum spatial $L^p$-spaces. We assume $M$ is a general von Neumann algebra (because the Hilsum $L^p$-spaces are defined for general von Neumann algebras) which admits a normal semifinite faithful weight $\varphi$ and furthermore, $M$ is represented on a Hilbert space $H$ and we have a normal semifinite faithful weight $\psi$ on the commutant $M^{\prime}$ of $M$.

  A vector $\xi \in H$ is said to be $\psi$-bounded if there exists a positive constant $c$ such that $\lgg y\xi\rgg ^{2} \leq c\psi(y^{\ast}y)$ for any $y \in n_{\psi}$, where $n_{\psi}=\{y\in M^{\prime}\ |\ \psi(y^{\ast}y)< \infty\}$. We let $D(H, \psi)$ be the subspace of $H$ consisting of all $\psi$-bounded vectors. Then for any $\xi \in D(H, \psi)$, $R^{\psi}(\xi)$ is the unique bounded operator from $H_{\psi}$ (the GNS representation Hilbert space induced by $\psi$) to $H$ such that $R^{\psi}(\xi)\eta_{\psi}(y)=y\xi$, where $\eta_{\psi}$ is the canonical injection of $n_{\psi}$ into $H_{\psi}$, and $\theta^{\psi}(\xi,\xi)=R^{\psi}(\xi)R^{\psi}(\xi)^{\ast}\in M$. The map $\xi \rightarrow \phi (\theta^{\psi}(\xi, \xi))$ defines a lower semicontinuous positive form on $D(H, \psi)$, where $\phi$ is any normal semifinite weight on $M$. Then the positive self-adjoint operator associated with this form is called the spatial derivative $\frac{d\phi}{d\psi}$ defined by Connes (\cite{6}).

  The Hilsum spatial $L^{p}$-space $L^{p}(M ; \psi)$ is defined as $(1 \leq p < \infty)$
  \begin{align*}
    L^{p}(M; \psi) &=\left\{\begin{array}{l} a \text{ is a closed
      densely defined operator on } H \text{
      with}\\
    \text{polar decomposition } a=u|a| \text{ such that }u\in M \\
    \text{ and } \left| a\right|^p=\frac{d\phi}{d\psi}\text{ for some
    }\phi\in M^+_*
    \end{array}\right\}\\
    &=\left\{\begin{array}{l} a \text{ is a closed densely defined
      operator on } H \text{ and}\\
    (-\frac{1}{p}) \text{-homogeneous with respect to } \psi \text{ such
      that }\\
    \int|a|^{p}d\psi < \infty
    \end{array}\right\}
  \end{align*}

  The $L^{p}$-norm is defined by $\lgg \cdot\rgg _{p}=(\int |\cdot|^{p} d\psi)^{\frac{1}{p}}$. If $p= \infty$, $L^{\infty}(M; \psi)= M$ with the usual operator norm.

  For the definition of homogeneity with respect to $\psi$ and the detailed properties of $\frac{d\phi}{d\psi}$ and $L^{p}(M ; \psi)$, we refer the reader to \cite{6,13,31}.

  Concerning the Hilsum spatial $L^{p}$-spaces, Terp's paper (\cite{31}) deeply revealed the close relation between the complex interpolation theory on $(M, M_{\ast})$ and $L^{p}(M;\psi)$.

  The subspace $L$ of $M$ consists of $x \in M$ for which there exists a $\varphi_{x}\in M_{\ast}$ such that $$\text{ for any } y, z \in n_{\varphi}: (z^{\ast}y, \varphi_{x})= \langle J\pi_{\varphi}(x)^{\ast}J\eta_{\varphi}(y), \eta_{\varphi}(z)\rangle,$$ where $J$ is the associated modular conjugation in $H_{\varphi}$, $(\cdot, \cdot)$ denotes the duality between $M$ and $M_{\ast}$, and $\langle\ , \rangle$ is the scalar inner product on $H_{\varphi}$ .

  For $x \in L$, we put $\lgg x\rgg _{L}=\max\{\lgg x\rgg , \lgg \varphi_{x}\rgg \}$, where the norm for $\varphi_{x}$ means the functional norm.

  The normed space $(L, \lgg \cdot\rgg _{L})$ is a Banach space and it can be embedded naturally into $M$ and $M_{\ast}$ by $x\mapsto x: L\rightarrow M$ and $x \mapsto \varphi_{x} : L\rightarrow M_{\ast}$.

  By transposition of the above two embeddings we have the injections $M\rightarrow L^{\ast}$ and $M_{\ast}\rightarrow L^{\ast}$ given by $$(x, y)_{(L, L^{\ast})}=(y, \varphi_{x})_{(M, M_{\ast})}, \; x\in L$$ for all $y \in M$ and $$(x, \phi)_{(L, L^{\ast})}=(x, \phi)_{(M, M_{\ast})}, \; x\in L$$ for all $\phi \in M_{\ast}$, where $L^{\ast}$ means the dual of $(L, \lgg \cdot\rgg _{L})$.

  The following diagram commutes,
  
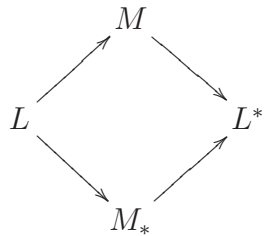
\[\xymatrix{  &{M}\ar[rd] \\
    {L}\ar[rd]\ar[ru]&&{L^{\ast}}\\
    &{M_{\ast}}\ar[ru]
  }\]
  and $L=M \cap M_{\ast}$ when $M$ and $M_{\ast}$ are considered as subspaces of $L^{\ast}$ (see Section 1, \cite{31}), and $L$ is $\sigma$-weakly dense in $M$, the embedding of $L$ in $M_*$ is weakly and norm dense in $M_*$ (Corollary 5, \cite{31}).

  Hence $(M, M_{\ast})$ is turned into a compatible pair of Banach spaces in the complex interpolation sense (Section 2.3, \cite{3}), and Terp proved that for $1<p<\infty$, the Hilsum spatial $L^{p}$-spaces $L^{p}(M; \psi)$ are just the complex interpolation spaces of $M_{\ast}$ and $M$. Accurately speaking, $L^{p}(M; \psi)=C_{\frac1p}(M, M_{\ast})$ (Theorem 36 in \cite{31}).

  Let $T$ be a linear map on $M$ satisfying the following conditions.

$(1)$ $T$ is a contraction of $M$, i.e., $\lgg Tx\rgg\leq \lgg x\rgg$ for all $x\in M$.

$(2)$ $T$ is positive, i.e., $Ty\geq 0$ if $y\in M_+$.

$(3)$ $\varphi(T(y))\leq \varphi(y)$ for all $y \in L_+$.

From now on, we will concentrate on the $\sigma$-finite cases.

Let $M$ be a $\sigma$-finite von Neumann algebra acting standardly on the Hilbert space $H$, $M$ admits a normal faithful state $\varphi$, $\psi$ is a normal faithful state on $M^\prime$, and $T$ is a linear transform satisfying the above conditions $(1)-(3)$.

We note that for any $x\in M$, we have $x\in m_\varphi =span\{x\in M_+\ |\ \varphi(x)<\infty\}$ as $\varphi$ is a state, and we know from Note (2), p.329 in \cite{31} that $m_\varphi \subseteq L$, combined with the fact that $L$ is defined to be a linear subspace of $M$, we have $m_\varphi=M=L$ in the $\sigma$-finite cases.

That is to say, condition $(3)$ above can be replaced by the following one when $M$ is $\sigma$-finite.

$(3)$ $\varphi(T(y))\leq \varphi(y)$ for all $y \in M_+$.

Moreover, the embedding of $M=L$ into $L^1(M; \psi)$ is given by the map $M \rightarrow d^{\frac12}Md^{\frac12}$, where $d=\frac{d\varphi}{d\psi}$ is the spatial derivative of $\varphi$ with respect to $\psi$ (see Section 2.3 and Theorem 27 in \cite{31}). Such an embedding is equivalent to the embedding $M\rightarrow M^\eta \subseteq M_*$ with $\eta=\frac12$ (the symmetric case) in Definition 7.1 of \cite{18} and $L^p(M; \psi)$ is equivalent to Kosaki's $L^p$-spaces $C_{\frac1p}(M^\eta, M_*)$ with $\eta=\frac12$ (Definition 7.2, \cite{18}). Moreover, since the Hilbert space is standard, the Hilsum space $L^p(M;\psi)$ is now as the same as the Araki-Masuda $L^p$-space in \cite{2}.

  Then for a general $p$, $1\leq p <\infty$, we may define the following map $T_p$ as
  \begin{align*}
    T_p:d^{\frac1{2p}}Md^{\frac1{2p}}&\rightarrow d^{\frac1{2p}}M\dtp\\
    d^{\frac1{2p}}xd^{\frac1{2p}}&\mapsto d^{\frac1{2p}}T(x)\dtp,
  \end{align*}
  for any $x\in M$.

 We claim that the map $T_1$ defined above extends naturally to a positive contraction of $L^1(M;\psi)\rightarrow L^1(M;\psi)$, and it will still be denoted by $T_1$. To show this claim, we need the Lemma 5.2 in \cite{11}. Although this lemma is stated and proved for Haagerup's $L^1$-spaces, it is still valid in the framework of the Hilsum spatial $L^1$-spaces through isometric isomorphism. We state it here for $L^1(M;\psi)$. Let $x\in M$ and $x$ is self-adjoint, then $$\lgg d^{\frac12}xd^{\frac12}\rgg_1=\inf\{\varphi(a)+\varphi(b)\ |\ x=a-b, a,b\in M_+\}.$$
 The proof of this result for $L^1(M; \psi)$ is essentially the same as that of Lemma 5.2 in \cite{11} when replacing $D$ there by the spatial derivative $d$. Then we will use this result and follow the method of Lemma 5.3 in \cite{11} to show our claim.
 Let $y\in M_+$, then $d^{\frac12}yd^{\frac12}\geq 0$, so $T_1(d^{\frac12}yd^{\frac12})=d^{\frac12}T(y)d^{\frac12}\geq 0$, hence $T_1$ is also positive. By condition $(3)$,
 \begin{align*}
 &\lgg d^{\frac12}T(y)d^{\frac12}\rgg_1=\int d^{\frac12}T(y)d^{\frac12}d\psi =\int T(y)dd\psi\\
 =&\varphi(T(y))\leq \varphi(y)=\lgg d^{\frac12}yd^{\frac12}\rgg_1.
 \end{align*}
 Now assume $x\in M$ and $x$ is self-adjoint, for any $\e>0$, there exist $a,b\in M_+$ such that $x=a-b$ and  $$\lgg d^{\frac12}ad^{\frac12}\rgg_1+\lgg d^{\frac12}bd^{\frac12}\rgg_1\leq\lgg d^{\frac12}xd^{\frac12}\rgg_1+\e .$$
 It follows that $$\lgg T_1(d^{\frac12}xd^{\frac12})\rgg_1\leq \lgg d^{\frac12}T(a)d^{\frac12}\rgg_1+\lgg d^{\frac12}T(b)d^{\frac12}\rgg_1\leq \lgg d^{\frac12}xd^{\frac12}\rgg_1 +\e .$$
 That is $\lgg T_1(d^{\frac12}xd^{\frac12})\rgg_1\leq\lgg d^{\frac12}xd^{\frac12}\rgg_1$ for any $x\in M$ is self-adjoint since $\e$ is arbitrary. Finally, decomposing any $x\in M$ into its real and imaginary parts, we get $\lgg T_1(d^{\frac12}xd^{\frac12})\rgg_1\leq 2\lgg d^{\frac12}xd^{\frac12}\rgg_1$. And since $d^{\frac12}Md^{\frac12}$ is norm dense in $L^1(M;\psi)$, $T_1$ extends to a bounded positive map on $L^1(M;\psi)$ with $\lgg T_1\rgg \leq 2$. Thus it remains to reinforce the norm bound $2$ to $1$. To this end, we consider the adjoint map $T_1^*$ which is a linear map on $M=L^1(M;\psi)^*$. Since $T_1$ is positive, $T_1^*$ is also positive and $T_1^*$ attains its norm at the identity $\mathbf1$ of $M$, i.e., $\lgg T_1\rgg=\lgg T_1^*\rgg=\lgg T_1^*(\mathbf1)\rgg$. Hence we are reduced to showing $\lgg T_1^*(\mathbf1)\rgg\leq 1$. Indeed, by condition $(3)$,
 $$\int T_1^*(\mathbf1)d^{\frac12}yd^{\frac12}d\psi=\int T_1(d^{\frac12}yd^{\frac12})d\psi=\varphi(T(y))\leq \varphi(y)=\lgg d^{\frac12}yd^{\frac12}\rgg_1$$ for any $y\in M_+$. We get $\lgg T_1^*(\mathbf1)\rgg\leq 1$ by the density of $d^{\frac12}M_+d^{\frac12}$ in $L^1_+(M; \psi)$, and hence $\lgg d^{\frac12}T(x)d^{\frac12}\rgg_1\leq \lgg d^{\frac12}xd^{\frac12}\rgg_1$ for any $x\in M$ and our claim follows.

  Combined with condition $(1)$ and the abstract Riesz-Thorin Theorem, we obtain that
  \begin{align*}
    \lgg \dtp T(x)\dtp\rgg_p \leq \lgg \dtp x\dtp\rgg_p, \hspace{1cm}\text{for }1<p<\infty, \text{ and } x\in M.
  \end{align*}
  Moreover, since for $1<p<\infty$, $L^p(M;\psi)=C_{\frac1p}(M_*,M)$ (Theorem 36 in \cite{31}, or Kosaki's noncommutative interpolation theorem in \cite{18}), in accordance with Theorem 4.2.2(a) in \cite{3}, we have that $\dtp M\dtp$ is $\lgg\cdot\rgg_p$-norm dense in $L^p(M;\psi)$.  Therefore, the map $T_p$ defined above also extends naturally to a positive contraction of $L^p(M;\psi)\to L^p(M;\psi)$  for each $p$ and we still denote it by $T_p$ $(1<p<\infty)$.  From the idea in \cite{31}, $M=L^\infty(M;\psi), M_*=L^1(M;\psi)$ and $L^p(M;\psi)$, $1<p<\infty$, can be regarded as injective subspaces in $M+M_*$ (factually it is just $M_*$ for the $\sigma$-finite cases, see Definition 7.2 in \cite{18}). In this situation it is easily seen that the maps $T,T_p\ (1<p<\infty)$ and $T_1$ defined above coincide on $L$.  For this reason, we may have a linear map on $M+M_*$, and the restriction of this map on $M$, $L^p(M;\psi)\ (1<p<\infty)$ and $M_*$ will be $T$, $T_p\ (1<p<\infty)$ and $T_1$, respectively when considering $M$, $L^p(M;\psi)\ (1<p<\infty)$ and $M_*$ in $M+M_*$.  Since this map on $M+M_*$ is deduced by $T$ on $M$, it is viewed as the extension of $T$ on $M+M_*$, and we still denote it by $T$.

  Then we arrive at the stage to show our main result. From above we know that the restriction of $T$ on $L^1(M;\psi)$ satisfies the following conditions.

 $\bullet$ \ $\lgg T(a)\rgg_1\leq \lgg a\rgg_1$, for all $a\in L_+^1(M;\psi)$.

 $\bullet$ \ $T$ is positive, i.e., $T(a)\geq 0$ if $a\in L^1_+(M;\psi)$.

  \begin{theorem}\label{sec:main-part-11}
   We assume $T$ is a linear transform on a $\sigma$-finite von Neumann algebra $M$ satisfying the conditions $(1)-(3)$ above, and $T$ extends to a linear positive contraction on $L^1(M;\psi)$.  Then for any $a \in L^{1}_{+}(M ; \psi)$ and any $\lambda > 0$, and any $n\in \mathbb N$, there exists a projection $e_n \in M$ such that
    \begin{align*}
    e_nS_r(a)e_n\leq \lam e_nde_n \text{ for all }r\in \{0,1,...,n\},
    \end{align*}
    and
    \begin{align*}
    \varphi(\mathbf 1-e_n)\leq \frac2{\lam}\int ad\psi,
    \end{align*}
    where $S_r(a)=\frac1{r+1}\sum_{k=0}^rT^k(a)$, $d=\frac{d\varphi}{d\psi}$ is the spatial derivative and $\mathbf 1$ is the identity of $M$.
    Furthermore, for any $a \in L^{1}_{+}(M ; \psi)$, there exists a projection $e\in M$ such that for any $r\in \mathbb N$,
    \begin{align*}
    \int e S_r(a)e d\psi \leq 4\lambda, \text{ and } \varphi(\mathbf 1-e)\leq \frac2{\lam}\int ad\psi.
    \end{align*}
  \end{theorem}
  \begin{proof}
    For any $x\in M_+$, $\int T(a)xd\psi$ is a positive linear functional on $L^1(M;\psi)$ since $T$ is positive. Indeed, the property $M=(L^1(M;\psi))^*$ implies that each $x\in M_+$ gives a positive linear functional on $L^1(M;\psi)$, and $T(a)$ is still positive, hence the functional action $\int T(a)xd\psi$ takes positive values for any $a \in L_+^1(M;\psi)$. Or in other words, we may also compute this value explicitly as follows.

    Since $x^{\frac{1}{2}} \in M_+$, in accordance with Proposition 8(4) in \cite{13}, $\lgg x^{\frac{1}{2}}T(a)\rgg_1 \leq \lgg x^{\frac{1}{2}}\rgg \lgg T(a)\rgg_1$, i.e., $x^{\frac{1}{2}}T(a) \in L^1(M; \psi)$. The Proposition in page 159 of \cite{13} yields that $\int T(a)x d\psi =\int x^{\frac{1}{2}}T(a)x^{\frac{1}{2}} d\psi$. For any family $\{\xi_\alpha\}$ in $D(H, \psi)$ such that $\sum_\alpha \theta^\psi(\xi_\alpha, \xi_\alpha)=\mathbf1$, hence from the definition of $\int \cdot d\psi$ in page 163 of \cite{6}, we have
    \begin{align*}
    &\int x^{\frac{1}{2}}T(a)x^{\frac{1}{2}}d\psi =\sum_\alpha \langle x^{\frac{1}{2}}T(a)x^{\frac{1}{2}}\xi_\alpha, \xi_\alpha\rangle\\
    =&\sum_\alpha \langle T(a)^{\frac{1}{2}}x^{\frac{1}{2}}\xi_\alpha, T(a)^{\frac{1}{2}}x^{\frac{1}{2}}\xi_\alpha\rangle=\sum_\alpha \lgg T(a)^{\frac{1}{2}}x^{\frac{1}{2}}\xi_\alpha \rgg^2 \geq 0.
    \end{align*}
    That is to say, $\int T(a)x d\psi \geq 0$ for any $a \in L^1_+(M; \psi)$ and $x\in M_+$.

    If the reader is more familiar with the properties of Haagerup's $L^p$-spaces, we may recall the Proposition 1.20 in \cite{10}, i.e., let $p, q \in [1, \infty]$, $\frac{1}{p}+\frac{1}{q}=1$ and let $a \in L^q(M)$, then $a\geq 0$ if and only if tr$(ab) \geq 0$ for any $b\in L^p(M)_+$, and Hilsum's spatial $L^p$-spaces are isometrically isomorphic to Haagerup's $L^p$-spaces, by replacing tr$(\cdot)$ by $\int \cdot d\psi$, we may also get the positivity of the functional at the beginning of this proof.

    Therefore there exists some $\wtilde x\in M_+$ such that $\int T(a)xd\psi=\int a\wtilde xd\psi$ again by the fact $M=(L^1(M;\psi))^*$.  If we denote $\wtilde x$ by $\wtilde T(x)$ for each $x\in M_+$, then
    \begin{align*}
      \int T(a)xd\psi=\int a\wtilde T(x)d\psi,
    \end{align*}
    and $\wtilde T$ extends linearly to be a linear transform on $M$ such that $\wtilde T(x)\geq 0$ if $x\geq 0$.

    Also we note that for any $x\in M_+$,
    \begin{align*}
      \lgg \wtilde T(x)\rgg &= \sup\left\{\int a\wtilde T(x)d\psi\ | \lgg a\rgg_1\leq 1, a\in L^1_+ (M;\psi)\right\}\\
      &=\sup\left\{\int T(a)xd\psi\ |\lgg a\rgg_1\leq 1,a\in L^1_+(M;\psi)\right\}\\
      &\leq \lgg x\rgg,
    \end{align*}
    where the last inequality is because $\lgg T(a)\rgg_1\leq \lgg a\rgg_1$ for any $a\in L^1_+(M;\psi)$. This result $\lgg \widetilde{T}(x)\rgg \leq \lgg x \rgg$ can also be obtained from the fact that $T$ is a contraction on $L^1(M;\psi)$ and $\widetilde{T}$ is just the adjoint map of $T$.

    Let $n\geq 1$ be fixed, we put
    \begin{align*}
    K=\{(x_0,x_1,\dotsc,x_n)\ |\ x_r\in M_+ \text{ for } 0\leq r \leq n \text{ and } \sum_{r=0}^nx_r\leq \mathbf 1\},
    \end{align*}
    then $K$ is $\si$-weakly compact in $M\times M\times\dotsb\times M$ of $(n+1)$-copies.

    We know from Section 2.3 in \cite{31} that the embedding map $\mu_1$ of $L_+$ into $L^1(M; \psi)_+$ is given by $\mu_1(x)=d^\frac{1}{2}xd^\frac{1}{2}=(x^\frac{1}{2}d^\frac{1}{2})^\ast(x^\frac{1}{2}d^\frac{1}{2})$ for $x\in L_+$, where $d=\frac{d\varphi}{d\psi}$ is the spatial derivative. Hence we may consider the value $\int d^\frac{1}{2}xd^\frac{1}{2}=\int\frac{d\varphi_x}{d\psi}d\psi=\varphi_x(\mathbf 1)<\infty$ since $\varphi_x$ is a positive normal linear functional on $M$. And for any $x\in M_+$, as $M=L$ in $\sigma$-finite case, $\int d^{\frac12}xd^{\frac12}d\psi<\infty$ for any $x\in M_+$ is well-defined.

    Therefore for any $a \in L_1^+(M; \psi)$, we may define a function $g$ on $K$ by
    \begin{align*}
    g((x_0, x_1, ..., x_n))=\sum_{r=0}^n (r+1)\int S_r(a)x_r d\psi-\lambda\sum_{r=0}^n (r+1)\int d^\frac{1}{2}x_rd^\frac{1}{2} d\psi.
    \end{align*}
    Note that $S_r(b)x_r \in L^1(M; \psi)$ as $x_r \in M$ and $S_r(b)\in L_+^1(M; \psi)$, so $\int S_r(b)x_r d\psi$ is well-defined and takes finite values for each $0\leq r\leq n$. We recall that $K$ is $\sigma$-weakly compact and $g$ is $\si$-weakly continuous on $K$, hence the finite maximum value of $g$ is attained for some choice of $(\overline{x_0},\overline{x_1},\dotsc,\overline{x_n})$ in $K$. Then we explain the reason that $g$ is $\sigma$-weakly continuous. For each $r\in \mathbb N\cup \{0\}$, there exists a $\phi_r\in M_*$ such that $\int S_r(a)x_r d\psi =\int \frac{d\phi_r}{d\psi}x_rd\psi=\phi_r(x_r)$ for $x_r\in M$, and $\phi_r$ is normal, i.e., $\phi_r$ is $\sigma$-weakly continuous on $M$. Also we have $\int d^{\frac12}x_rd^{\frac12}d\psi=\int dx_r d\psi=\int \frac{d\varphi}{d\psi}x_rd\psi=\varphi(x_r)$, for $x_r\in M$, is $\sigma$-weakly continuous on $M$ for $\varphi$ is a normal state on $M$. Thus $g$ is $\sigma$-weakly continuous on $K$.

    We let $\mathbf 1-\sum_{r=0}^n\overline{x_r}=z_n$ (we choose this notation because the positive operator $z_n$ depends on $n$).  For any $x\in M_+$ with $x\leq z_n$, we have $\sum_{r=0}^n\overline{x_r}+x\leq \mathbf 1$ and $\sum_{r=0}^n \overline{x_r}+x\in M_+$ and hence for any fixed $r_0\in\{0,1,\dotsc,n\}$,
    \begin{align*}
      g((\overline{x_0},\overline{x_1},\dotsc,\overline{x_n}))\geq g((\overline{x_r}+\de(r,r_0)x)_{r=0,1,\dotsc,n}), \text{ where }\de(r,r_0)=
      \begin{cases}
        1 & r=r_0\\
        0 & r\neq r_0.
      \end{cases}
    \end{align*}
    As a result, we have
    \begin{align*}
      &\sum_{r=0}^n (r+1)\int S_r(a)\overline{x_r}d\psi-\lambda\sum_{r=0}^n (r+1)\int d^\frac{1}{2}\overline{x_r}d^\frac{1}{2}d\psi\\
      \geq& \sum_{r=0}^n (r+1)\int S_r(a)\overline{x_r}d\psi+(r_0+1)\int S_{r_0}(a)xd\psi-\\
      &-\lambda\sum_{r=0}^n (r+1)\int d^\frac{1}{2}\overline{x_r}d^\frac{1}{2}d\psi-\lambda(r_0+1)\int d^\frac{1}{2}xd^\frac{1}{2} d\psi.
    \end{align*}
    Then we get
    \begin{align*}
      (r_0+1)\int S_{r_0}(a)xd\psi\leq\lambda (r_0+1)\int d^\frac{1}{2}xd^\frac{1}{2}d\psi \text{ for any }r_0\in\{0,1,\dotsc,n\}.
    \end{align*}
    From the above inequality, for any $x\in M_+$, $x\leq z_n$, any $r$ in $\{0,1,\dotsc,n\}$, we have
    \begin{equation}
      \int S_r(a)xd\psi\leq \lam\int d^\frac{1}{2}xd^\frac{1}{2}d\psi. \label{eq:1-zq-prop11-1}
    \end{equation}

    Take $y=(y_r)_{r=0,1,\dotsc,n}$ with
    \begin{align*}
      y_n=
      \begin{cases}
        \wtilde T(\overline{x_{r+1}}) & 0\leq r\leq n-1,\\
        0 & r=n,
      \end{cases}
    \end{align*}
    and we have $\sum_{r=0}^{n-1}\wtilde T(\overline{x_{r+1}})\leq \mathbf 1$.
  Indeed, because $\wtilde T$ is linear, $$\sum_{r=0}^{n-1}\wtilde T(\overline{x_{r+1}})=\wtilde T(\sum_{r=0}^{n-1}\overline{x_{r+1}})\leq \wtilde T(\mathbf 1),$$ and for any $a\in L_+^1(M;\psi)$, we have
    \begin{align*}
      \int a\wtilde T(\mathbf 1)d\psi=\int T(a)d\psi=\lgg T(a)\rgg_1\leq \lgg a\rgg_1.
    \end{align*}
    Therefore $\wtilde T(\mathbf 1)\leq \mathbf 1$ and we get $\sum_{r=0}^{n-1}\wtilde T(\overline{x_{r+1}})\leq \mathbf 1$, so $y$ is in $K$.

    As a result, we obtain that  $g(\overline{x_0}, \overline{x_1},\dotsc,\overline{x_n})\geq g(y)$.  That is to say,
    \begin{align*}
      &\sum_{r=0}^n (r+1)\int S_r(a)\overline{x_r}d\psi -\lambda\sum_{r=0}^{n}(r+1)\int d^\frac{1}{2}\overline{x_r}d^\frac{1}{2}d\psi\\
      \geq& \sum_{r=0}^{n-1} (r+1)\int T(S_r(a))\overline{x_{r+1}}d\psi-\lambda\sum_{r=0}^{n-1} (r+1)\int d^\frac{1}{2}\wtilde T(\overline{x_{r+1}})d^\frac{1}{2}d\psi.
    \end{align*}
    Hence from the above inequality, it follows
    \begin{align}
     &\sum_{r=0}^n (r+1)\int S_r(a)\overline{x_r}d\psi-\sum_{r=0}^{n-1}(r+1)\int T(S_r(a))\overline{x_{r+1}}d\psi \nonumber \\
       \geq & \lambda\sum_{r=0}^n(r+1)\int d^\frac{1}{2}\overline{x_r}d^\frac{1}{2}d\psi-\lambda\sum_{r=0}^{n-1} (r+1)\int d^\frac{1}{2}\wtilde
      T(\overline{x_{r+1}})d^\frac{1}{2}d\psi.  \label{eq:1-zq-prop11-3}
    \end{align}
   We compute the left hand side of (\ref{eq:1-zq-prop11-3}), and it equals
    \begin{align*}
      &\sum_{r=0}^n (r+1)\int S_r(a)\overline{x_r}d\psi-\sum_{r=0}^{n-1} (r+1)\int T(S_r(a))\overline{x_{r+1}}d\psi\\
      =& \int S_0(a)\overline{x_0}d\psi+\sum_{r=1}^n (r+1)\int S_r(a)\overline{x_r}d\psi -\sum_{r=0}^{n-1} (r+1)\int T(S_r(a))\overline{x_{r+1}}d\psi\\
      =& \int S_0(a)\overline{x_0}d\psi+\sum_{r=0}^{n-1}(r+2)\int S_{r+1}(a)\overline{x_{r+1}}d\psi-\sum_{r=0}^{n-1} (r+1)\int T(S_r(a))\overline{x_{r+1}}d\psi.
    \end{align*}

    Since $(r+2)S_{r+1}(a)=a+T(a)+\dotsb+T^r(a)+T^{r+1}(a)$ and $(r+1)TS_r(a)=T(a)+T^2(a)+\dotsb+T^r(a)+T^{r+1}(a)$, the left hand side of (\ref{eq:1-zq-prop11-3}) equals
    \begin{align*}
      &\int S_0(a)\overline{x_0}d\psi+\sum_{r=0}^{n-1}\int a\overline{x_{r+1}}d\psi\\
      =& \int a\overline{x_0}d\psi+\sum_{r=1}^n\int a\overline{x_r}d\psi=\sum_{r=0}^n\int a\overline{x_r}d\psi.
    \end{align*}
    Therefore, inequality (\ref{eq:1-zq-prop11-3}) becomes
    \begin{align*}
      \sum_{r=0}^n\int a\overline{x_r}d\psi\geq \lambda\sum_{r=0}^n (r+1)\int d^\frac{1}{2}\overline{x_r}d^\frac{1}{2}d\psi-\lambda\sum_{r=0}^{n-1} (r+1)\int d^\frac{1}{2}\wtilde T(\overline{x_{r+1}})d^\frac{1}{2}d\psi.
    \end{align*}
    So we obtain the following inequality,
    \begin{align*}
      &\sum_{r=0}^n\int a\overline{x_r}d\psi-\lambda\sum_{r=0}^n\int d^\frac{1}{2}\overline{x_r}d^\frac{1}{2}d\psi\\
      \geq& \lambda\sum_{r=0}^n r\int d^\frac{1}{2}\overline{x_r}d^\frac{1}{2}d\psi-\lambda\sum_{r=1}^n r\int d^\frac{1}{2}\wtilde T(\overline{x_r})d^\frac{1}{2}d\psi\geq 0,
    \end{align*}
    from the following inequalities,
    \begin{align*}
    &\int d^{\frac12}\wtilde{T}(\overline{x_r})d^{\frac12}d\psi\\
    =&\sup\{\int d^{\frac12}\wtilde{T}(\overline{x_r})d^{\frac12}xd\psi\ |\ x\in M_+, \lgg x\rgg\leq 1\}\\
    =&\sup\{\int x_rT(d^{\frac12}xd^{\frac12})d\psi\ |\ x\in M_+, \lgg x\rgg\leq 1\}\\
    =&\sup\{\int d^{\frac12}\overline{x_r}d^{\frac12}T(x)d\psi\ |\ x\in M_+, \lgg x\rgg\leq 1\}\\
    \leq&\sup\{\int d^{\frac12}\overline{x_r}d^{\frac12}xd\psi\ |\ x\in M_+, \lgg x\rgg\leq 1\}\\
    =&\int d^{\frac12}\overline{x_r}d^{\frac12}d\psi.
    \end{align*}

    Hence we have
    \begin{equation}
      \sum_{r=0}^n\int a\overline{x_r}d\psi\geq \lam \sum_{r=0}^n\int d^\frac{1}{2}\overline{x_r}d^\frac{1}{2}d\psi, \label{eq:1-zq-prop11-4}
    \end{equation}
    for any $a\in L_+^1(M;\psi)$.

    Letting $\mathbf 1 -\sum ^n_{r=0} \overline{x_r}=\int^1_0 s dp_s$ be the spectral decomposition of $\mathbf 1 -\sum ^n_{r=0} \overline{x_r}$, suppose $y \in M_+$, $y \leq \mathbf 1$. Writing $y_m=(\mathbf 1-p_{m^{-1}})y(\mathbf 1-p_{m^{-1}})$ for each $m\in \mathbb N$, we have
    \begin{align*}
    0\leq y_m \leq \mathbf 1-p_{m^{-1}} \leq m(\mathbf 1 - \sum^n_{r=0}\overline{x_r}).
    \end{align*}
    Hence by (\ref{eq:1-zq-prop11-1}), $$\int S_r(a)y_m d\psi\leq \lambda \int d^{\frac12}y_md^{\frac12} d\psi$$ for $0 \leq r \leq n$.

    Taking the limit as $m\rightarrow \infty$, and putting $e_n= \mathbf 1- p_0$, we get
    \begin{align*}
    \int e_n S_r(a)e_n y d\psi &=\int S_r(a)e_n y e_n d\psi\\
    &\leq \lambda \int d^{\frac12}e_n y e_n d^{\frac12} d\psi\\
    &=\lambda \int e_n d e_n y d\psi,
    \end{align*}
    hence $e_n S_r(a)e_n \leq \lambda e_n d e_n$ for $0\leq r \leq n$ because $M=(L^1(M;\psi))^*$.

   Therefore for $r= 0$, we get $e_n a e_n \leq \lambda e_n d e_n$, and since $\sum^n_{r=0}\overline{x_r} \in M_+$, we have
   \begin{align*}
   \int a e_n (\sum^n_{r=0}\overline{x_r}) d\psi\leq \lambda \int d e_n (\sum^n_{r=0}\overline{x_r})d\psi,
   \end{align*}
   which together with (\ref{eq:1-zq-prop11-4}) gives
   \begin{align*}
   \int a (\mathbf 1 -e_n) (\sum^n_{r=0}\overline{x_r}) d\psi\geq \lambda \int d (\mathbf 1-e_n) (\sum^n_{r=0}\overline{x_r})d\psi.
   \end{align*}
   Since $p_0\mathbf 1 -p_0\sum ^n_{r=0} \overline{x_r}=p_0\int^1_0 s dp_s=0$, we obtain that $\mathbf 1-e_n = (\mathbf 1-e_n)(\sum^n_{r=0}\overline{x_r})$.
   Hence we get
   \begin{align*}
   \int d^{\frac12}(\mathbf 1-e_n)d^{\frac12}d\psi\leq \frac1{\lambda}\int a (\mathbf 1-e_n)d\psi\leq \frac1{\lambda}\int a d\psi.
   \end{align*}

   From the above procedure, we get $e_n$ for each $n\in \ds N$ such that $\varphi(\mathbf 1-e_n)\leq \frac1{\lam}\int ad\psi$ and $e_nS_r(a)e_n\leq \lam e_n d e_n$ for any $r$ in $\{0,1,\dotsc,n\}$.  Choose a subnet $e_{n_k}$ which converges weakly to some $h\in M$ with $0\leq h\leq \mathbf 1$.

   We assume that $S_r(a)\in L^1_+(M;\psi)$ corresponds to some $\phi_r\in M_\ast^+$ for each $r\in \mathbb N$. Since $M$ acts standardly on $H$ and each $\phi_r$ is normal and hence a vector state, for any fixed $r\in \mathbb N$, there exists a vector $\xi_{r}$ in $H$ such that $\phi_r=\omega_{\xi_r,\xi_r}$. As a result,
   \begin{align*}
   \int hS_r(a)hd\psi=\phi_r(h^2)=\langle h^2\xi_{r}, \xi_{r}\rangle=\lgg h\xi_{r}\rgg^2.
   \end{align*}
   The subnet $e_{n_k}$ converges to $h$ in the weak operator topology, and for each $r\in \mathbb N$, we have $e_{n_k}\xi_{r}\rightarrow h\xi_{r}$ in the weak topology of $H$ when $k\rightarrow\infty$. Since the Hilbert space norm $\lgg\cdot\rgg$ is lower semicontinuous relative to this topology, we have
   \begin{align}
   \lgg h\xi_{r}\rgg^2 \leq \liminf_{k\rightarrow \infty}\lgg e_{n_k}\xi_{r}\rgg^2.
   \end{align}

  Hence, we have
  \begin{align*}
  \int hS_r(a)hd\psi \leq \liminf_{k\rightarrow\infty}\lgg e_{n_k}\xi_{r}\rgg^2.
  \end{align*}

  Combined with
  \begin{align*}
  &\lgg e_{n_k}\xi_{r}\rgg^2=\phi_r(e_{n_k})=\int e_{n_k}S_r(a)e_{n_k}d\psi\\
  \leq&\lambda\int e_{n_k}de_{n_k}d\psi=\lambda\int e_{n_k}\frac{d\varphi}{d\psi}d\psi=\lambda\varphi(e_{n_k}),
  \end{align*}
  for $k\in \mathbb N$ such that $n_k$ is larger than $r$, we conclude that
  \begin{align*}
  \int hS_r(a)h d\psi\leq \lambda\lim_{k\rightarrow \infty}\varphi(e_{n_k})=\lambda\varphi(h),
  \end{align*}
  since $\varphi$ is a normal state on $M$, i.e., $\varphi$ is weakly continuous on the unit ball of $M$.

  Taking the spectral decomposition for $h=\int_0^1sde_s$, and let $e=\mathbf1 -e_{\frac12}\in M$ for $e_{\frac12}$ is in $M$, let $g=\int_{\frac12}^1s^{-1}de_s$. Then for each $r\in \mathbf N$,
   \begin{align*}
     \int eS_r (a)e d\psi = \int gh S_r(a)hg d\psi \leq 4\int h S_r(a)h d\psi,
   \end{align*}
   thanks to $e=gh$ and $\lgg g \rgg \leq 2$.

   Therefore we obtain that
   \begin{align*}
   \int eS_r(a)ed\psi \leq 4\int h S_r(a)h d\psi \leq 4\lambda\varphi(h)\leq 4\lambda\varphi(\mathbf1)=4\lambda.
   \end{align*}

   Moreover, for the reason $\mathbf 1-e=e_{\frac12}\leq 2(\mathbf 1-h)$, we get
   \begin{align*}
   \varphi(\mathbf1-e)\leq 2\varphi(\mathbf1-h)=2\lim_{k\rightarrow \infty}\varphi(\mathbf1-e_{n_k})=2\lim_{k\rightarrow\infty}\int d^{\frac12}(\mathbf1-e_{n_k})d^{\frac12}d\psi\leq \frac2\lambda \int a d\psi,
   \end{align*}
   also because $\varphi$ is a normal state on $M$.
  \end{proof}
 We should point out that we learned much from Professor Kosaki and Professor Xu in the proof of this theorem.

 We turn to the case that $M$ is semifinite, $L^{1}(M ; \psi)$ is equivalent to the tracial $L^{1}(M)$. Moreover, since the construction of the Hilsum $L^p$-spaces is independent on the choice of the normal semifinite faithful weight $\psi$ in the isometrically isomorphic sense, hence we may choose a special case that $\psi(\cdot)=\varphi(J\cdot J)$, then the spatial derivative $\frac{d\varphi}{d\psi}$ becomes $\triangle_{\varphi}$, i.e., the modular operator associated with $\varphi$ and in the semifinite case, $\varphi$ can be assumed to be a normal faithful semifinite trace. Considering the modular automorphism group induced by a trace is trivial, we get $\triangle_{\varphi}=\mathbf 1$ in this case, then the result in the above theorem is for any $a \in L^{1}_{+}(M ; \psi)$ and any $\lambda > 0$, and any $n\in \mathbb N$, there exists a projection $e_n \in M$ such that
 $e_nS_r(a)e_n\leq \lam e_n \text{ for all }r\in \{0,1,...,n\}$ and $\varphi(\mathbf 1-e_n)\leq \frac2{\lam}\int ad\psi$, then by Yeadon's method in \cite{32}, $S_r(a)^{\frac12}e_{n_k}\rightarrow S_r(a)^{\frac12}h$ weakly as $k\rightarrow \infty$. In fact, for $\xi_1\in H,\xi_2 \in D(S_r(a)^{\frac 12})$, we have $\langle S_r(a)^{\frac12}e_{n_k}\xi_1, \xi_2\rangle=\langle e_{n_k}\xi_1, M_r(a)^{\frac12}\xi_2\rangle\to \langle h\xi, S_r(a)^{\frac12}\xi_2\rangle$, and $$\ldg \langle h\xi_1,S_r(a)^{\frac12}\xi_2\rangle\rdg=\lim_{k\to \infty}\ldg\langle S_r(a)^{\frac12}e_{n_k}\xi_1, \xi_2\rangle\rdg\leq \lam^{\frac 12}\lgg \xi_1\rgg \lgg\xi_2\rgg.$$ Since $\lgg S_r(a)^{\frac12}e_{n_k}\rgg=\lgg e_{n_k}S_r(a)e_{n_k}\rgg^{\frac12}\leq \lam^{\frac12}$ if $n_k\geq r$, so that $h\xi_1\in D(S_r(a)^{\frac12})$ and
 $$\langle S_r(a)^{\frac12}h\xi_1, \xi_2\rangle=\langle h\xi_1, S_r(a)^{\frac12}\xi_2\rangle=\lim_{k\to \infty}\langle S_r(a)^{\frac12}e_{n_k}\xi_1, \xi_2\rangle.$$ Hence it follows
\begin{align*}
&\langle h S_r(a)h\xi_1, \xi_1\rangle=\lgg S_r(a)^{\frac12}h\xi_1\rgg^2 \leq \varliminf_{k\to\infty}\lgg S_r(a)^{\frac12}e_{n_k}\xi_1\rgg^2 \\
=&\varliminf_{k\to\infty}\langle e_{n_k}S_r(a)e_{n_k}\xi_1, \xi_1\rangle \leq \lim_{k\to \infty}\lam\langle e_{n_k}\xi_1, \xi_1\rangle=\lam\langle h\xi_1, \xi_1\rangle,
\end{align*}
 i.e., $h S_r(a)h\leq \lam h$ for each $r\in \ds N$. Taking the spectral decomposition for $h=\int^1_0 s de_s$, and let $e=\mathbf 1-e_{\frac12}$, $g=\int^1_{\frac12}s^{-1}de_s$. Then we have $eS_r(a)e\leq \lambda ghg\leq 2\lambda e$, this is just the result of Yeadon's Theorem 1 in \cite{32}.

 Then we introduce the conceptions of ``type'' and ``weak type'' for the action of $T$ on $M+M_*$.  Such conceptions ``type'' and ``weak type'' appeared in the classical real analysis first and have been widely used in classical ergodic theorems.  They are modified by Junge-Xu (pp 396--397, \cite{17}), for the framework of noncommutative $L^p$-spaces associated with a semifinite von Neumann algebra.  Here we rewrite them for the transform $T$ we constructed above in the framework of $L^p(M;\psi)$'s $(1\leq p\leq \infty)$.

  For each $n\in\ds N$, $S_n$ is a linear map on $M+M_*$ satisfying the conditions $(1)-(3)$.  Thus $S=(S_n)_{n\in\ds N}$ is a map which sends a positive element in $L^p(M;\psi)$ for some fixed $p$, $1\leq p\leq \infty$ to a sequence of positive elements in $L^p(M;\psi)$.  Here we identify $L^1(M;\psi)$ with $M_*$ and $L^\infty(M;\psi)$ with $M$.

  We say that $S$ is of type $(p,p),(1\leq p\leq \infty)$ if there is a positive constant $c$ such that for any $x\in L_+^p(M;\psi)$, there is $a\in L^p_+(M;\psi)$ satisfying $\lgg a\rgg_p\leq c\lgg x\rgg_p$ and $S_n(x)\leq a$, for any $n\in \ds N$.

  Then we say that $S$ is of weak type $(p,p), (1\leq p<\infty$), if there is a positive constant $c$ such that for any $x\in L^p_+(M;\psi)$, and any $\lam>0$ there is a projection $e\in M$ such that
  \begin{align*}
    \varphi(\mathbf 1-e)\leq \left(\frac{c\lgg x\rgg_p}{\lam}\right)^p \text{ and } eS_n(x)e\leq \lam\mathbf 1, \text{ for any } n\in\ds N,
  \end{align*}
   where $\mathbf 1$ is the identity of the von Neumann algebra $M$.

Yeadon's Theorem shows that $S=(S_r)_{r\in \mathbb N}$ is of weak type $(1,1)$ when $M$ is semifinite.

It is obvious that $S=(S_r)_{r\in \mathbb N}$ is of type $(\infty,\infty)$ for an arbitrary von Neumann algebra $M$. Indeed, by condition (1), we get that for each $r\in\ds N$, if $x\in M_+=L^\infty_+(M;\psi)$, we have $T^r(x)\geq 0$ for $n\in\ds N$ and $\lgg T^r(x)\rgg\leq \lgg x\rgg$ and thus $\lgg S_r(x)\rgg\leq \lgg x\rgg$. Hence if we put $a=\lgg x\rgg \mathbf 1\in L^\infty_+(M;\psi)$, we have $\lgg a\rgg=\lgg x\rgg$ and $S_r(x)\leq a$, for all $r\in\ds N$.

As we have mentioned in the previous section, such a weak type conception is no longer appropriate for the non-tracial cases. We give a pre-version of pre-weak type here, we hope it is of some meaning.

We say that $S$ is of pre-weak type $(p,p), (1\leq p<\infty)$ if there is a positive constant $c$ such that for any $x\in L^p_+(M;\psi)$, and any $\lam>0$ there is a projection $e\in M$ such that
  \begin{align*}
    \varphi(\mathbf 1-e)\leq \left(\frac{c\lgg x\rgg_p}{\lam}\right)^p \text{ and } \lgg eS_n(x)e\rgg _p\leq \lam, \text{ for any } n\in\ds N,
  \end{align*}
   where $\mathbf 1$ is the identity of the von Neumann algebra $M$. Theorem \ref{sec:main-part-11} shows that $S=(S_r)_{r\in \mathbb N}$ is of pre-weak type $(1,1)$ for a $\sigma$-finite von Neumann algebra.

If we would like to obtain a satisfactory non-tracial extension of Theorem \ref{sec:introduction-thm1.1}, one possible method is to consider real interpolation theory (possible non-tracial real interpolation theory for pre-weak type, though we are not very sure about the existence of such theory and it is still in exploring) of von Neumann algebras, because the real interpolation theory always provides us the type for midpoints from the weak type assumption of endpoints.  But as is pointed out by Junge-Xu (pp 385--386, \cite{17}), in contrast with the classical theory, the noncommutative nature of weak type $(1,1)$ inequalities seems a priori unsuitable for classical interpolation arguments.  More accurately speaking, Pisier-Xu gave a counterexample saying that the complex interpolation space $L^p(M)$ may not coincide with the real interpolation space $L^{p,p}(M)$ for non-tracial von Neumann algebras if we establish non-commutative real interpolation theory verbatim from classical one (p.1472, Example 3.3, \cite{25}).  We had tried several ways to modify the definition of real interpolation construction in order to suit well to the von Neumann algebra theory and the complex interpolation of Terp (\cite{31}) at the same time.  Unfortunately, we have not obtained any valid method for this matter so far. Let us point out that a key obstacle in this work is the absense of generalized singular numbers for the non-tracial cases.  The generalized singular number function (\cite{FK}) is a powerful tool when dealing with the $\tau$-measurable operators associated with a semifinite von Neumann algebra, and the noncommutative tracial $L^p$-spaces are just consisting of such operators.  In the non-tracial cases, whether there exists such a counterpart theory which contains the generalized singular number theory by \cite{FK} has not been sufficiently understood now. We do not know whether there is any other method available and we will continue to explore this problem in the future.

Finally, we give three examples as applications of Theorem \ref{sec:main-part-11}.

  \begin{example}
    Let $(\Omega,\mathfrak{F},\mu)$ be a finite measure space and $N$ be a $\sigma$-finite von Neumann algebra equipped with a normal faithful state $\varphi_1$. We consider the von Neumann algebra tensor product $(M,\varphi_2)=(L^\infty(\Omega),\mu)\overline{\otimes}(N,\varphi_1)$, where $\varphi_2$ is a normal faithful state since $\mu$ is finite and $\varphi_1$ is a normal faithful state. For $1\leq p<\infty$, the corresponding noncommutative $L^p(M;\psi_2)$ is just $L^p(\Omega, L^p(N;\psi_1))$, the usual $L^p$-space of strongly measurable $p$-integrable functions on $\Omega$ with values in $L^p(N;\psi_1)$, where $\psi_1$ (resp. $\psi_2$) is a normal faithful state on the commutant of $N$ (resp. $M$), and we may choose $\psi_1$ (resp. $\psi_2$) to be associated with $\varphi_1$ (resp. $\varphi_2$) by the Tomita-Takesaki theory. Now let $S$ be a linear map on $L^\infty(\Omega)$ satisfying conditions $(1)-(3)$ (with $M=L^\infty(\Omega)$ there), then $T=I\otimes S$ is a linear map on $M$ verifying the same conditions (with $M= L^\infty (\Omega)\overline{\otimes} N$ there).
    From Theorem \ref{sec:main-part-11}, for any $a \in L^{1}(\Omega,L_+^1(N; \psi_1))$ and any $\lambda > 0$, and any $n\in \mathbb N$, there exists a projection $e_n \in M$ such that
    \begin{align*}
    e_nS_r(a)e_n\leq \lam e_nde_n \text{ for all }r\in \{0,1,...,n\},
    \end{align*}
    and
    \begin{align*}
    \varphi_2(\mathbf 1-e_n)\leq \frac2{\lam}\int ad\psi_2,
    \end{align*}
    where $S_r(a)=\frac1{r+1}\sum_{k=0}^rI\otimes S^k(a)$, $d=\frac{d\varphi_2}{d\psi_2}$ is the spatial derivative and $\mathbf 1$ is the identity of $M$.
    Furthermore, for any $a \in L^{1}(\Omega,L_+^1(N; \psi_1))$, there exists a projection $e\in M$ such that for any $r\in \mathbb N$,
    \begin{align*}
    \int e S_r(a)e d\psi_2 \leq 4\lambda, \text{ and } \varphi_2(\mathbf 1-e)\leq \frac2{\lam}\int ad\psi_2.
    \end{align*}
  \end{example}

  \begin{example}
    Let $M$ be a von Neumann algebra with a normal faithful state $\varphi$, and let $N$ be any von Neumann subalgebra of $M$. The generalized conditional expectation $\e: M\rightarrow N$ relative to $\varphi$ defined by Accardi-Cecchini is given as $\e(x)=J_NP_NJ\pi_\varphi(x)JP_NJ_N$ for any $x\in M$ (see \cite{1}). If we regard $\e$ to be a linear map from $M$ to $M$, $\e$ satisfies conditions $(1)-(3)$. Therefore, Theorem \ref{sec:main-part-11} implies that for any $a \in L^{1}_{+}(M ; \psi)$ and any $\lambda > 0$, and any $n\in \mathbb N$, there exists a projection $e_n \in M$ such that
    \begin{align*}
    e_nS_r(a)e_n\leq \lam e_nde_n \text{ for all }r\in \{0,1,...,n\},
    \end{align*}
    and
    \begin{align*}
    \varphi(\mathbf 1-e_n)\leq \frac2{\lam}\int ad\psi,
    \end{align*}
    where $S_r(a)=\frac1{r+1}\sum_{k=0}^rT^k(a)$, $d=\frac{d\varphi}{d\psi}$ is the spatial derivative and $\mathbf 1$ is the identity of $M$.
    Furthermore, for any $a \in L^{1}_{+}(M ; \psi)$, there exists a projection $e\in M$ such that for any $r\in \mathbb N$,
    \begin{align*}
    \int e S_r(a)e d\psi \leq 4\lambda, \text{ and } \varphi(\mathbf 1-e)\leq \frac2{\lam}\int ad\psi.
    \end{align*} If $N$ is globally invariant under the modular automorphism group $\sigma_t^\varphi$, then $\e$ will be the conditional expectation in the sense of \cite{T}, i.e., a projection of norm one. As a projection is idempotent, we have $e_n\e(a)e_n\leq \lam e_nd e_n$ and $\int e\e(a)ed\psi \leq 4\lambda$ in this case.
  \end{example}

  \begin{example}
    Let $\{(M_i, \varphi_i)\}_{i \in I}$ be a family of von Neumann algebras and each $\varphi_i$ is a normal faithful state. Let $(M, \varphi)=*_{i \in I}(M_i, \varphi_i)$ be the von Neumann algebra reduced free product (see \cite{NS}), and hence $\varphi$ is a normal faithful state on $M$. Put $M_i^\circ=\{x\in M_i | \varphi_i(x)=0\}$, then $M_i=\mathbb{C}\mathbf{1}_{M_i}\oplus M_i^\circ$, and let $T_i: M_i\rightarrow M_i$ be defined by
$$T_i{|}_{\mathbb{C}\mathbf{1}_{M_i}}=I_{\mathbb{C}\mathbf{1}_{M_i}} \text{  and  } T_i{|}_{M_i^\circ}=\exp(-1)I_{M_i^\circ},$$
and $\{T_i\}_{i\in I}$ defines a positive linear map $T$ on $M$ by free product, and $T$ is uniquely determined by its action on monomials:
\begin{align*}
T(x_1x_2\cdot\cdot\cdot x_n)=T_{i_1}(x_1)T_{i_2}(x_2)\cdot\cdot\cdot T_{i_n}(x_n)=\exp(-n)x_1x_2\cdot\cdot\cdot x_n,
\end{align*}
for any $x_1, x_2, \cdot\cdot\cdot, x_n$ with $x_k \in M_{i_k}^\circ$ and $i_1\neq i_2\neq \cdot\cdot\cdot \neq i_n$. The map $T$ is called the free product of the family $\{T_i\}_{i\in I}$. Then $T$ satisfies conditions $(1)-(3)$ with respect to $M$ since each $T_i$ satisfies conditions $(1)-(3)$ with respect to $M_i$. Hence $T$ extends to a positive linear map on $L^1(M; \psi)$ and by Theorem \ref{sec:main-part-11}, for any $a \in L^{1}_{+}(M ; \psi)$ and any $\lambda > 0$, and any $n\in \mathbb N$, there exists a projection $e_n \in M$ such that
    \begin{align*}
    e_nS_r(a)e_n\leq \lam e_nde_n \text{ for all }r\in \{0,1,...,n\},
    \end{align*}
    and
    \begin{align*}
    \varphi(\mathbf 1-e_n)\leq \frac2{\lam}\int ad\psi,
    \end{align*}
    where $S_r(a)=\frac1{r+1}\sum_{k=0}^rT^k(a)$, $d=\frac{d\varphi}{d\psi}$ is the spatial derivative and $\mathbf 1$ is the identity of $M$.
    Furthermore, for any $a\in L^1_+(M; \psi)$, there exists a projection $e\in M$ such that
    \begin{align*}
    \int e S_r(a)e d\psi \leq 4\lambda, \text{ for any } r\in \mathbb N, \text{ and } \varphi(\mathbf 1-e)\leq \frac2{\lam}\int ad\psi.
    \end{align*}
  \end{example}

\pagebreak
\noindent\textbf{Acknowledgements}\\\\
\indent The author would like to take this opportunity to express his highest respect and appreciation to his supervisor Professor Y.Kawahigashi, who led him into the realm of von Neumann algebra theory and paid great care and patience for him in the past five years. And the author is also very grateful to Professor H.Kosaki in Kyushu University and Professor Q.Xu in Universit\'{e} de Franche-Comt\'{e} from whom he learned many techniques on noncommutative $L^p$-spaces through e-mail correspondence, and they kindly pointed out mistakes in the previous version of this paper. And he also thanks for Professor N.Ozawa and Professor Y.Ogata and other ones in their group, since he benefited greatly from seminars and discussions in the group. Finally, his thanks go to Mr. Zhao in his research room, for much help from Mr. Zhao.\\

\end{document}